\documentclass{article}
\usepackage{amssymb}
\usepackage{amsfonts}
\usepackage{amsmath}
\usepackage{color}
\usepackage{cancel}

\newtheorem{theorem}{Theorem}

\newtheorem{definition}[theorem]{Definition}
\newtheorem{example}[theorem]{Example}

\newtheorem{lemma}[theorem]{Lemma}

\newenvironment{proof}[1][Proof]{\noindent\textbf{#1.} }{\ \rule{0.5em}{0.5em}}
\begin{document}

\title{An Existence Result for a Stochastic Stefan Problem With Mushy Region and Turbulent Transport Noise}
\author{Ioana CIOTIR\footnote{Normandie University, INSA de Rouen Normandie, LMI (EA 3226 – FR CNRS 3335), 76000 Rouen, France, {\tt\small ioana.ciotir@insa-rouen.fr}}
\and
Franco FLANDOLI\footnote{Scuola normale superiore di Pisa, Piazza dei Cavalieri, 7, 56126 Pisa PI, Italie {\tt\small franco.flandoli@sns.it}}
\and
Dan GOREAC\footnote{\'{E}cole d'Actuariat, Universit\'{e} Laval, QC G1V0B3, Qu\'{e}bec, Canada {\tt\small dan.goreac@act.ulaval.ca} and LAMA, Univ Gustave Eiffel, UPEM, Univ Paris Est Creteil, CNRS, F-77447 Marne-la-Vall\'{e}e, France, {\tt\small dan.goreac@univ-eiffel.fr}}
}
\maketitle

\begin{abstract}
This work is devoted to the proof of the existence of a martingale solution for a complex version of the stochastic Stefan problem. This particular formulation incorporates two important features: a mushy region and turbulent transport within the liquid phase. While our approach bears similarities to porous media equations, it differs in a crucial aspect. Instead of using the typical framework for such equations, we have chosen to work within an $L^2$ space. This choice is motivated by the nature of the operator that characterizes the turbulent noise in our model. The $L^2$ space provides a more natural and appropriate setting for handling this specific operator, allowing us to better capture and analyze the turbulent transport phenomena in the liquid phase of the Stefan problem.
\end{abstract}



\section{Introduction}

\bigskip

\subsection{Construction of the model}

\bigskip
In the present work, we investigate a free boundary problem that describes the melting/solidification process in a turbulent fluid, in the presence of a heating source generically denoted by $F$. We consider a two-phase Stefan problem set on a bounded open domain $\mathcal{O} \subset \mathbb{R}^d$, where $d = 1, 2,$ or $3$. This domain is assumed to have a smooth boundary. A crucial aspect of our model is the phase-transition region, which is characterized as a \emph{mushy} zone, connecting the solid and liquid phases.\\

We shall first consider the following sharp Stefan-type problem:%
\begin{equation}
\left\{ 
\begin{array}{ll}
C_{1}\dfrac{\partial \theta }{\partial t}-{div}\left( k_{1}\nabla
\theta \right) +u\cdot \nabla \eta \left( \theta \right) =F,\medskip & \quad 
\text{if }\theta <0, \\ 
C_{2}\dfrac{\partial \theta }{\partial t}-{div}\left( k_{2}\nabla
\theta \right) +u\cdot \nabla \eta \left( \theta \right) =F,\medskip & \quad 
\text{if }\theta >0, \\ 
\left( k_{2}\nabla \theta ^{+}-k_{1}\nabla \theta ^{-}\right) \cdot N_{\xi
}^{-}=l\cdot N_{t},\medskip & \quad \text{on }\{\theta =0\}, \\ 
\theta ^{+}=\theta ^{-}=0,\medskip & \quad \text{on }\{\theta =0\}, \\ 
\theta \left( 0,\xi \right) =\theta _{0}\left( \xi \right) ,\medskip & \quad 
\text{in }\mathcal{O}, \\ 
\theta \left( t,\xi \right) =0, & \quad \text{on }\partial \mathcal{O\times }%
\left( 0,T\right) ,%
\end{array}%
\right.  \label{eq1}
\end{equation}%
where $\theta ^{+}$, resp. $\theta ^{-}$ are the right, resp. left limits of the
\emph{free boundary} situated between the solid and the liquid phase, $N\left(
t,\xi \right) $ is the \emph{unit normal to the interface} and $l$ is assumed to be
the \emph{latent heat}.

We denote by $k_{1}$ and $k_{2}$ the \emph{thermal conductivity} of the solid and
liquid phases, and by $C_{1}$ and $C_{2}$ the \emph{specific heat} of the two
phases.\par
The function $\eta :\mathbb{R\rightarrow R}$ is assumed to be a smooth function which vanishes in the solid phase, and such
that $\eta \left( 0\right) =0$. We assume that $\left\vert \eta ^{\prime }\left(
r\right) \right\vert \leq L,$ for $\forall r\in \mathbb{R}$. From the physical point of view, it is coherent to strengthen the null-behavior, and further assume that \[\eta \left(
\theta \right) =0,\textnormal{ for }0<\theta <\varepsilon.\]
Since the term $u\cdot \nabla \eta \left( \theta \right) $ is meant to model the
\emph{turbulence} present in the liquid phase, the physical interpretation of $\eta 
$ is that the solid phase is not allowed to move and neither is a small
liquid region close to the boundary.

Concerning the \emph{velocity} of the fluid, a reasonable model for the turbulent fluid
has the form 
\begin{equation*}
u\left( t,\xi \right) =\sum_{k=1}^{\infty }\alpha _{k}\sigma _{k}(\xi)\frac{%
d\beta _{k}(t)}{dt},
\end{equation*}%
where 
\begin{itemize}
    \item $\left\{ \alpha _{k}\right\} _{k}$ is a sequence of positive constants
conveniently chosen; 
\item $\left\{ \sigma _{k}\right\} _{k}$ is a sequence of divergence-free
smooth vector fields whose properties will be defined later on, and
\item $\left\{ \beta
_{k}\right\} _{k}$ is a sequence of independent Brownian motions.
\end{itemize}

Interpreted in the \emph{Stratonovich sense}, the turbulence has the following formulation%
\begin{equation*}
u\left( t,\xi \right) =\sum_{k=1}^{\infty }\alpha _{k}\sigma _{k}(\xi)\circ
d\beta _{k}(t),
\end{equation*}%
which leads to the following explicit formulation for the turbulence%
\begin{equation}
u\cdot \nabla \eta \left( \theta \right) =\sum_{k=1}^{\infty }\alpha
_{k}\sigma _{k}\cdot \nabla \eta \left( \theta \right) \circ d\beta
_{k} .  \label{noise}
\end{equation}
The Stratonovich integrals are used, in accordance with the Wong-Zakai principle (see rigorous results in \cite{Debussche}). As mentioned above, the heuristic idea is that turbulence can appear in the liquid region of the phase change problem due to the difference in temperature between the two phases.

This type of noise has been introduced and intensively studied during the recent years. The reader is invited to refer to  \cite{Flandoli-book}, \cite{Flandoli5}, \cite{Flandoli1}, \cite{Flandoli2}, \cite{Flandoli-Pappalettera}, among other references.

\bigskip

\subsection{The Rigorous Construction of the Noise}
In order to study the equation, we make the following rigorous construction
of the turbulence term. Let $\left\{ e_{k}\right\} _{k}$\ be a complete
orthonormal basis in $L^{2}\left( \mathcal{O}\right) $ formed by the
eigenfunctions of the Dirichlet homogeneous Laplace operator on $\mathcal{O}$%
, with $\left\{ \lambda _{k}\right\} _{k}$ being the corresponding eigenvalues,
i.e., 
\begin{equation*}
-\Delta e_{k}=\lambda _{k}e_{k},\quad \forall k\in \mathbb{N}^*\text{.}
\end{equation*}

\noindent The reader is recalled that, by standard estimates, there exists a constant $C>0$ such that $\left\vert e_{k}\right\vert _{\infty }\leq C\lambda _{k}$, and that the constant $C$ is independent of $k$.

We further consider $\left\{ \mu _{k}\right\} _{k}$ to be a sequence of divergence-free
vectors belonging to $\left( C^{\infty }\left( \mathcal{O}\right) \right) ^{d}$, and such
that $${div}\left( \mu _{k}e_{k}\right) =0$$.

\medskip

We define the operator 
\begin{equation*}
B:H_{0}^{1}\left( \mathcal{O}\right) \rightarrow L_{2}\left( L^{2}\left( 
\mathcal{O}\right) ;L^{2}\left( \mathcal{O}\right) \right)
\end{equation*}%
where%
\begin{eqnarray*}
B\left( \theta \right) &:&L^{2}\left( \mathcal{O}\right) \rightarrow
L^{2}\left( \mathcal{O}\right) \\
B\left( \theta \right) \left( \varphi \right) &=&\sum_{k=1}^{\infty }\alpha
_{k}\sigma _{k}\cdot \nabla \eta \left( \theta \right) \left( e_{k},\varphi
\right) _{2}
\end{eqnarray*}%
where $\left\{ \alpha _{k}\right\} _{k}$ is a sequence of real values and $%
\left\{ \sigma _{k}\right\} _{k}$ is a sequence of vectors of the type 
\begin{equation*}
\sigma _{k}=\mu _{k}e_{k},\quad \forall k=\overline{1,\infty }.
\end{equation*}

Under suitable assumptions mentioned below, one can see that $B$ is well
defined from $H_{0}^{1}\left( \mathcal{O}\right) $ to the space of
Hilbert-Schmidt operators $L_{2}\left( L^{2}\left( \mathcal{O}\right)
;L^{2}\left( \mathcal{O}\right) \right) $.  

We have%
\begin{eqnarray*}
\left\Vert B\left( \theta \right) \right\Vert _{L_{2}\left( L^{2}\left( 
\mathcal{O}\right) ;L^2\left( \mathcal{O}\right) \right) }^{2}
&=&\sum_{k=1}^{\infty }\left\vert \alpha _{k}\sigma _{k}\cdot \nabla \eta
\left( \theta \right) \right\vert _{2}^{2} \\
&=&\sum_{k=1}^{\infty }\left\vert \alpha _{k}\right\vert ^{2}\left\vert
e_{k}\left\langle \mu _{k},\nabla \eta \left( \theta \right) \right\rangle _{%
\mathbb{R}^{d}}\right\vert _{2}^{2} \\
&\leq&\sum_{k=1}^{\infty }\left\vert \alpha _{k}\right\vert ^{2}\left\vert
\left\vert e_{k}\mu _{k}\right\vert _{\mathbb{R}^{d}}\left\vert \eta
^{\prime }\left( \theta \right) \nabla \theta \right\vert _{\mathbb{R}%
^{d}}\right\vert _{2}^{2} \\
&\leq &L^2C^2\sum_{k=1}^{\infty }\left\vert \alpha _{k}\right\vert
^{2}\left\vert \lambda _{k}\right\vert ^{2}\left\vert \mu _{k}\right\vert
_{\left( L^{\infty }\left( \mathcal{O}\right) \right) ^{d}}^{2}\left\vert
\left\vert \nabla \theta \right\vert _{\mathbb{R}^{d}}\right\vert _{2}^{2} \\
&\leq &L^2C^2\sum_{k=1}^{\infty }\left\vert \alpha _{k}\right\vert
^{2}\left\vert \lambda _{k}\right\vert ^{2}\left\vert \mu _{k}\right\vert
_{\left( L^{\infty }\left( \mathcal{O}\right) \right) ^{d}}^{2}\left\vert
\theta \right\vert _{H_{0}^{1}\left( \mathcal{O}\right) }^{2}.
\end{eqnarray*}

\noindent We consider a \emph{cylindrical Wiener process} in $L^{2}\left( \mathcal{O}\right) $
constructed with respect to the aforementioned basis
\begin{equation*}
dW\left( t\right) =\sum_{k=1}^{\infty }e_{k}d\beta _{k}\left( t\right)
\end{equation*}%
where $\left\{ \beta
_{k}\right\} _{k}$ is a sequence of independent standard Brownian motions on a
filtered probability space $\left( \Omega ,\mathcal{F},\left( \mathcal{F}%
_{t}\right) _{t\geq 0},\mathbb{P}\right) .$

As we have already mentioned, the noise term of the equation is meant in the Stratonovich sense, i.e., (\ref{noise}) can be
written as 
\begin{equation*}
B\left( \theta \right) \circ dW\left( t\right) =\sum_{k=1}^{\infty }\alpha
_{k}\sigma _{k}\cdot \nabla \eta \left( \theta \right) \circ d\beta
_{k}\left( t\right),
\end{equation*}%
which means that the noise (\ref{noise}) can be written as%
\begin{equation*}
u\cdot \nabla \eta \left( \theta \right) =B\left( \theta \right) \circ
dW\left( t\right) .
\end{equation*}

In order to study the equation, we can transform the Stratonovich
integral into a It\^{o} one, whichever is more convenient for differential formulations:
\begin{eqnarray*}
B\left( \theta \right) \circ dW\left( t\right) &=&\sum_{k=1}^{\infty }\alpha
_{k}\sigma _{k}\cdot \nabla \eta \left( \theta \right) d\beta _{k}\left(
t\right) \\
&&-\frac{1}{2}\sum_{k=1}^{\infty }\alpha _{k}^{2}{div}\left[ \left(
\eta ^{\prime }\left( \theta \right) \right) ^{2}\sigma _{k}\otimes \sigma
_{k}\nabla \theta \right] dt.
\end{eqnarray*}%
The corresponding It\^{o} integral has the following form%
\begin{equation*}
B\left( \theta \right) dW\left( t\right) =\sum_{k=1}^{\infty }\alpha
_{k}\sigma _{k}\cdot \nabla \eta \left( \theta \right) d\beta _{k}\left(
t\right) ,
\end{equation*}%
followed by a correction term.

Again, by assuming that $\left\{ \alpha
_{k}\right\} _{k}$ contains elements that are small enough, we are able to denote by%
\begin{equation}\label{Q}
Q\left( \xi \right) =\sum_{k=1}^{\infty }\alpha _{k}^{2}\left( \sigma
_{k}\left( \xi \right) \otimes \sigma _{k}\left( \xi \right) \right) .
\end{equation}The elements $\sigma_k$ are now considered small enough to guarantee an absolute convergence in the previous expression.

In order to facilitate the reading of this paper, we denote the matrix
operator $Q=\left( q_{i,j}\right) _{1\leq i,j\leq d}$ where each $q_{i,j}$
is a series in $k$ which converges under the assumptions mentioned at the end
of the introduction.

On the other hand, we introduce the real function 
\begin{equation*}
g\left( x\right) =\frac{1}{2}\int_{0}^{x}\left( \eta ^{\prime }\left(
r\right) \right) ^{2}dr,\quad \forall x\in \mathbb{R}\text{,}
\end{equation*}%
and we rewrite the noise as%
\begin{equation}\label{B_Qg}
B\left( \theta \right) \circ dW\left( t\right) =\sum_{k=1}^{\infty }\alpha
_{k}\sigma _{k}\cdot \nabla \eta \left( \theta \right) d\beta _{k}\left(
t\right) -{div}\left[ Q\nabla g\left( \theta\right) %
\right] dt.
\end{equation}

\begin{example}
\label{example1}Let us consider an example showing how a sequence of functions $\left\{ \mu _{k}\right\} _{k}$
can be constructed such that the associated $\sigma_k$ obey to the previously-introduced assumptions. To this purpose, we let $\mathcal{O\subset }\mathbb{R}^{3}$, we denote by $\xi=(\xi_1,\xi_2,\xi_3)$ the associated components, and set \[%
\mu _{k}=\left( \frac{\partial}{\partial _{\xi_2}}e_{k},-\frac{\partial}{\partial _{\xi_1}}e_{k},0\right), \] for each $%
k\in \mathbb{N}^*$. The divergence condition is then easily checked. Furthermore, owing to the boundary behavior of the elements of the basis $e_k$, it follows that \[N_\xi\cdot \sigma_k=0.\]
This is particularly useful to get rid of boundary contributions when employing Green's formula.\\
We need to assume, in addition to the Hypotheses above, that
the elements of the sequence $\left\{ \alpha _{k}\right\} _{k}$ are small
enough such that the following series are convergent, for every $i,j\in\{1,\ldots,d\}$,%
\begin{equation*}
\sum_{k=1}^{\infty }\alpha _{k}^{2}\left( \frac{\partial }{\partial _{\xi
_{i}}}e_{k}^{2}\right) \mu _{k}^{i}\mu _{k}^{j}+\sum_{k=1}^{\infty }\alpha
_{k}^{2}e_{k}^{2}\left( \frac{\partial }{\partial _{\xi _{i}}}\mu
_{k}^{i}\right) \mu _{k}^{j}+\sum_{k=1}^{\infty }\alpha _{k}^{2}e_{k}^{2}\mu
_{k}^{i}\left( \frac{\partial }{\partial _{\xi _{i}}}\mu _{k}^{j}\right)
<\infty .
\end{equation*}

By tempering with the bounds on the coefficients $(\alpha _{k})_{k}$, 
one further guarantees that the assumptions \eqref{Ip1} and \eqref{Ip2} below are satisfied.
\end{example}
\bigskip

\subsection{The Main Assumptions}

Throughout this paper we shall assume that the elements of the sequence $%
\left\{ \alpha _{k}\right\} _{k}$ are small enough such that the following
series are converging

\begin{itemize}
\item For the well-posedness of the operator $B$ from the noise we assume
that%
\begin{equation}
\sum_{k=1}^{\infty }\left\vert \alpha _{k}\right\vert ^{2}\left\vert \lambda
_{k}\right\vert ^{2}\left\vert \mu _{k}\right\vert _{\left( L^{\infty
}\left( \mathcal{O}\right) \right) ^{d}}^{2}\leq C_{1}<\infty ,  \label{Ip1}
\end{equation}%
for some constant $C_{1}$.

\item For the well-posedness of the operator $Q=\left( q_{ij}\right)
_{1\leq i,j\leq d}$ (with each $q_{ij}$ a series in $k$) which appears in
the It\^{o}-Stratonovich correction term, we assume that 
\begin{equation}
\sum_{k=1}^{\infty }\alpha _{k}^{2}\left( \sigma _{k}\left( \xi \right)
\otimes \sigma _{k}\left( \xi \right) \right) \ \text{are convergent for almost every $\xi\in\mathcal{O}$.}
\label{Ip2}
\end{equation}
\item For the well-posedness of the problem we further assume that%
\begin{equation}
\gamma =\underset{i,j=\overline{1,d}}{\max }\left\{ \left\vert
q_{ij}\right\vert _{\infty }+\left\vert \frac{\partial q_{ij}}{\partial \xi
_{i}}\right\vert _{\infty }\right\} <\infty .  \label{Ip3}
\end{equation}%
and the matrix $Q$ is positively defined, i.e. 
\begin{equation*}
a^{t}Qa\geq 0,\quad \forall a\in \mathbb{R}^{d}\text{.}
\end{equation*}
Although clear enough, let us mention that $\mid{\cdot}\mid_\infty$ denotes the $L^\infty(\mathcal{O})$ norm.
\item We assume that $\widetilde{\gamma}$ and its inverse are smooth and $0\leq \left( \widetilde{\gamma }^{-1}\right) ^{\prime
}\leq 1$.

\end{itemize}

\subsection{The Mathematical Model}

Under the assumptions above we can rigorously write
the system (\ref{eq1}) as follows
\begin{equation}
\left\{ 
\begin{array}{ll}
C_{1}d\theta -{div}\left( k_{1}\nabla \theta \right) dt-{div}\left[
Q \nabla g\left( \theta \right) \right] dt+B\left( \theta
\right) dW\left( t\right) =F,\medskip & \text{if }\theta <0, \\ 
C_{2}d\theta -{div}\left( k_{2}\nabla \theta \right) dt-{div}\left[
Q \nabla g\left( \theta \right) \right] dt+B\left( \theta
\right) dW\left( t\right) =F,\medskip & \text{if }\theta >0, \\ 
\left( k_{2}\nabla \theta ^{+}-k_{2}\nabla \theta ^{-}\right) \cdot N_{\xi
}^{-}=l\cdot N_{t},\medskip & \text{on }\theta =0, \\ 
\theta ^{+}=\theta ^{-}=0,\medskip & \text{on }\theta =0, \\ 
\theta \left( 0,\xi \right) =\theta _{0}\left( \xi \right) ,\medskip & \text{%
in }\mathcal{O}, \\ 
\theta \left( t,\xi \right) =0, & \text{on }\partial \mathcal{O\times }%
\left( 0,T\right) .%
\end{array}%
\right.  \label{eq2}
\end{equation}

In order to study the existence and uniqueness of a solution and the system
above, we shall first write it as a nonlinear multivalued problem of
monotone type%
\begin{equation}
\left\{ 
\begin{array}{ll}
d\gamma \left( \theta \right) -{div}\left( k\left( \theta \right)
\nabla \theta \right) dt-{div}\left( Q \nabla g\left(
\theta \right) \right) dt+B\left( \theta \right) dW\left( t\right)
=F,\medskip & \mathcal{O\times }\left( 0,T\right) , \\ 
\theta \left( 0,\xi \right) =\theta _{0}\left( \xi \right) ,\medskip & 
\mathcal{O}, \\ 
\theta \left( t,\xi \right) =0, & \partial \mathcal{O\times }\left(
0,T\right) .%
\end{array}%
\right.  \label{eq3}
\end{equation}%
where $\gamma \left( r\right) =C\left( r\right) +l\times H\left( r\right) $
with 
\begin{equation*}
C\left( r\right) =\left\{ 
\begin{array}{cc}
C_{1}r, & r\leq 0, \\ 
C_{2}r, & r>0.%
\end{array}%
\right.
\end{equation*}%
The function $H$ is a Heaviside contribution. Furthermore, we consider  
\begin{equation*}
k\left( r\right) =\left\{ 
\begin{array}{cc}
k_{1}, & r\leq 0, \\ 
k_{2}, & r>0.%
\end{array}%
\right.
\end{equation*}
At this point, by formally using the change of variable ${\gamma }\left( \theta \right) =X$,\
we can rewrite the equation above as%
\begin{equation}
\left\{ 
\begin{array}{ll}
\begin{array}{l}
dX_0-\Delta \Psi_0 \left( X_0\right) dt-{div}\left( Q
\nabla g\left( {\gamma }^{-1}\left( X_0\right) \right) \right)
dt\medskip \\ 
\quad \quad \quad \quad \quad \quad \quad \quad \quad \quad \quad \quad
+B\left( {\gamma }^{-1}\left( X_0\right) \right) dW\left( t\right)
=F,\medskip%
\end{array}
& \mathcal{O\times }\left( 0,T\right) , \\ 
X_0\left( 0,\xi \right) ={\gamma }\left( \theta _{0}\right) \left(
\xi \right) \overset{not}{=}x,\medskip & \mathcal{O}, \\ 
X_0\left( t,\xi \right) =0, & \partial \mathcal{O\times }\left( 0,T\right) ,%
\end{array}%
\right.  \label{eq4}
\end{equation}%
where 
\begin{equation*}
\Psi_0 \left( r\right) =\left\{ 
\begin{array}{ll}
k_{1}C_{1}^{-1}r,\medskip & r\leq 0, \\ 
0,\medskip & r\in (0,l),\\
k_{2}C_{2}^{-1}(r-l) , & r\geq l.%
\end{array}%
\right.
\end{equation*}

Since, from a physical point of view, at the interface between ice and water
is not sharp, it is reasonable to assume the presence of a \emph{mushy region}
which is mathematically take into account by replacing the Heaviside function with a smoothed one
denoted by $\widetilde{H}$. With the same reasoning, we  replace the function $\gamma $ with the
corresponding $\widetilde{\gamma }$. This is done in order for the corresponding $\Psi$ to be \emph{strictly monotone} and of class $C^1$.

The new equation is 
\begin{equation}
\left\{ 
\begin{array}{ll}
\begin{array}{l}
dX-\Delta \Psi \left( X\right) dt-{div}\left( Q
\nabla g\left( \widetilde{\gamma }^{-1}\left( X\right) \right) \right)
dt\medskip \\ 
\quad \quad \quad \quad \quad \quad \quad \quad \quad \quad \quad \quad
+B\left( \widetilde{\gamma }^{-1}\left( X\right) \right) dW\left( t\right)
=F,\medskip%
\end{array}
& \mathcal{O\times }\left( 0,T\right) , \\ 
X\left( 0,\xi \right) =\widetilde{\gamma }\left( \theta _{0}\right) \left(
\xi \right) \overset{not}{=}x,\medskip & \mathcal{O}, \\ 
X\left( t,\xi \right) =0, & \partial \mathcal{O\times }\left( 0,T\right) ,%
\end{array}%
\right.  \label{eq4}
\end{equation}%
where 
$\Psi$ is a smoothening version (obtained, for instance, by reasonable convolution/mollification) in such a way that the function $\Psi $ is null at $0$, strictly monotone, which means that there
exist a positive constant $\psi _{0}$ such that%
\begin{equation*}
\left( \Psi \left( x\right) -\Psi \left( y\right) \right) \left( x-y\right)
\geq \psi _{0}\left\vert x-y\right\vert ^{2},\quad \forall x,y\in \mathbb{R}%
\text{.}
\end{equation*}

Throughout the paper, we denote by $C$ a generic constant which is independent of the
approximation parameters, and which is allowed to change from one line to another.\\

The reader is invited to note that, from a technical point of view, the main difficulty of this problem comes
from the presence of the Stratonovich-type turbulent noise which is well
posed in $L^{2}$, while the rest of the equation is well posed in $H^{-1}$, where the monotonicity of the $\Delta\Psi$ operator is usually satisfied.\\

In the present work we give an existence result of a martingale solution
for the equation by using an $L^{2}$ setting, and our result covers the physically-motivated model explained above.\\

Most of the previously existing results concerning the porous media type equation, asserted in
an $H^{-1}$ or $L^{1}$ setting, concern fundamentally different frameworks, and do not include divergence-type noise. The reader is referred to \cite{PM} for the general results and \cite{barbu-ln}, \cite{eu-superfast}, 
\cite{eu-dan-reika}, \cite{gess-soc} for some critical cases.
Some cases involving divergence-type noise for porous media were studied
also in $H^{-1}$ under stronger assumptions, in particular pertaining to the \emph{linearity} of the noise coefficient, and by using different technical tools. To our best knowledge, the existing literature in this direction reduces to \cite{turra} and \cite{eu-ito-div}. 

In comparison with previous results, e.g., \cite{eu-ito-div}, in our case the noise is of \emph{non-linear} nature, thus constituting a further difficulty, besides its higher relevance for physical models. As a requirement stemming from the non-linear nature of the noise coefficient, we construct the noise itself (through the $\sigma$ requirements) in a somewhat different way than the one in \cite{eu-ito-div}.
On the other hand, the result in \cite{eu-ito-div} provides the existence of
a martingale solution, which is also weak from a variational point of view,
for the porous media equation with still linear divergence It\^{o}-type noise in $H^{-1}$.

The result in
\cite{turra} in Stratonovich case, imposes the presence in the equation of
an extra term having a Laplace form which is assumed to dominate completely the
noise in the sense that it is assumed to be large enough to compensate the
influence of the noise term, and to further compensate the influence of the It\^{o}-Stratonovich
correction term. As a consequence of this rather involved and strong assumption, the mentioned reference can
treat the equation in an $H^{-1}$ setting (as one usually does for porous-media), and to get a stochasticaly and variationally strong
solution. Perhaps it is worth mentioning that while we do require strict monotonicity of the underlying porous-media operator $\Delta\Psi$, no assumption is made on the lower bounds of the monotonicity constant designed (as it was the case in \cite{turra}) to dominate the noise contribution. 

~\\


\bigskip

\section{Existence of the solution}

\bigskip

This section, constituting the core of our mathematical contribution, aims at providing the existence of the solution for the
equation (\ref{eq4}), solution formulated in the following sense.

\bigskip

\begin{definition}\label{DefSol}
Let $x\in L^{2}\left( \mathcal{O}\right) .$ We say that equation (\ref{eq4})
has a weak solution if 
there exist 
\begin{enumerate}
    \item a filtered reference probability space $\left(
\Omega ,\mathcal{F},\left( \mathcal{F}_{t}\right) _{t\geq 0},\mathbb{P}\right) $,
\item 
a sequence of independent $\mathcal{F}_{t}$\ Brownian motions $\left\{ \beta
_{k}\right\} _{k}$, and 
\item an $H^{-1}\left( \mathcal{O}\right) -$valued
continuous $\mathcal{F}_{t}-$adapted process $X$ such that 
\begin{enumerate}
    \item $X\in L^{2}\left(
\Omega \times \left( 0,T\right) \times \mathcal{O}\right) $, and 
\item the following holds true \begin{eqnarray*}
\left( X\left( t\right) ,e_{j}\right) _{2} &=&\left( x,e_{j}\right)
_{2}+\int_{0}^{t}\int_{\mathcal{O}}F(s)e_jd\xi ds+\int_{0}^{t}\int_{\mathcal{O}}\Psi \left( X\left( s\right) \right)
\Delta e_{j}d\xi ds \\
&&\int_{0}^{t}\int_{\mathcal{O}}g\left( \widetilde{\gamma }%
^{-1}\left( X\right) \right) {div}\left[ Q\nabla
e_{j}\right] d\xi ds \\
&&+\sum_{k=1}^{\infty }\int_{0}^{t}\alpha _{k}\left( \eta \left( \widetilde{%
\gamma }^{-1}\left( X\right) \right) ,\sigma _{k}\cdot \nabla e_{j}\right)
_{2}d\beta _{k}\left( s\right) ,
\end{eqnarray*}%
$\mathbb{P}$-a.s., for all $t\in[0,T]$, and for all $j\in \mathbb{N}^*$, where $\left\{ e_{j}\right\} $ is the
orthonormal basis in $L^{2}\left( \mathcal{O}\right) $ as introduced before.
\end{enumerate}
\end{enumerate}
\end{definition}

\bigskip
We now come to the main result of the paper.
\begin{theorem}\label{ThMain}
For each $x\in L^{2}\left( \mathcal{O}\right) $, there exists a solution to
the equation (\ref{eq4}), in the sense of the Definition \ref{DefSol}, and such that \[
\Psi \left( X\right) \in L^{2}\left( \Omega \times \left( 0,T\right) \times 
\mathcal{O}\right) .\]
\end{theorem}

\bigskip

\begin{proof}[Proof of Theorem \ref{ThMain}]
We begin with noting that the source $F$ can be considered regular enough, such that its presence does not impact the mathematical arguments. As a consequence, and without loss of generality, throughout the proof, we assume $F=0$.\\
We consider the Gelfand triple $V\subset H\subset V^{\ast }$ where $%
V=H_{0}^{1}\left( \mathcal{O}\right) ,$ $H=L^{2}\left( \mathcal{O}\right) $
and $V^{\ast }=H^{-1}\left( \mathcal{O}\right) .$

We let $H_{n}=span\left\{
e_{1},e_{2},...,e_{n}\right\} $, such that $span\left\{ e_{i}\left\vert i\in 
\mathbb{N}\right\vert \right\} $ is dense in $V$.

Let $P_{n}:V^{\ast }\rightarrow H_{n}$ be defined by%
\begin{equation*}
P_{n}y=\sum\limits_{i=1}^{n}\left( y,e_{i}\right) _{H}e_{i},\quad y\in
V^{\ast }.
\end{equation*}

Clearly, $\left. P_{n}\right\vert _{H}$ is just the orthogonal projection
onto $H_{n}$ in $H$. We take as initial data $X^{\left( n\right) }\left(
0\right) =P_{n}x.$ For (coherence and) notation purposes, we will also employ $x^{(n)}=P_{n}x$ to denote the initial datum for the approximating solutions.

For each finite $n\in \mathbb{N}$ we consider the following equation in $%
H_{n}$%
\begin{eqnarray*}
dX &=&P_{n}\Delta \Psi \left( P_{n}X\right) dt+P_{n}\textit{div}%
\left[ P_{n}\left(Q \nabla g\left( \widetilde{\gamma }%
^{-1}\left( P_{n}X\right) \right)\right) \right] dt \\
&&-P_{n}B\left( \widetilde{\gamma }^{-1}\left( P_{n}X\right) \right)
dW\left( t\right) 
\end{eqnarray*}%
which has a unique strong solution in finite dimension since the operators
are Lipschitz-continuous
\begin{eqnarray*}
dX^{\left( n\right) } &=&\Delta P_{n}\Psi \left( X^{\left( n\right) }\right)
dt+P_{n}\textit{div}\left[ P_{n}\left(Q\nabla g\left( 
\widetilde{\gamma }^{-1}\left( X^{\left( n\right) }\right) \right)\right) \right] dt
\\
&&-\sum\limits_{k=1}^{n }P_{n}\left( \alpha _{k}\sigma _{k}\cdot\nabla
\eta \left( \widetilde{\gamma }^{-1}\left( X^{\left( n\right) }\right)
\right) \right)d\beta _{k}\left( t\right) ,
\end{eqnarray*}%
where $P_{n}X=X^{\left( n\right) }$.
Local existence of the solution is a classical fact due to the regularity of the coefficients appearing in the equation, see for example \cite{K-S}, \cite{Prevot}, \cite{Skorokhod} (since the projected spaces are both pertaining to $L^2(\mathcal{O})$ and $H_0^1(\mathcal{O})$, we will not distinguish between them).

We first apply the It\^{o} formula in $L^{2}\left( \mathcal{O}\right) $ and
we get%
\begin{eqnarray*}
&&\left\Vert X^{\left( n\right) }\left( t\right) \right\Vert
_{2}^{2}+2\int_{0}^{t}\int_{\mathcal{O}}\Psi ^{\prime }\left( X^{\left(
n\right) }\right) \left\vert \nabla X^{\left( n\right) }\right\vert ^{2}d\xi
ds \\
&&+2\int_{0}^{t}\int_{\mathcal{O}}\left( \nabla \left( \left( g\circ 
\widetilde{\gamma }^{-1}\right) \left( X^{\left( n\right) }\right) \right)
\right)^t Q \nabla X^{\left( n\right) }d\xi ds \\
&=&\left\Vert X^{\left( n\right) }\left( 0\right) \right\Vert
_{2}^{2}+\sum\limits_{k=1}^{n}\int_{0}^{t}\int_{\mathcal{O}%
}\left\vert P_{n}\left( \alpha _{k}\sigma _{k}\cdot \nabla \left( \left(
\eta \circ \widetilde{\gamma }^{-1}\right) \left( X^{\left( n\right)
}\right) \right) \right) \right\vert ^{2}d\xi ds \\
&&+2\sum\limits_{k=1}^{n}\int_{0}^{t}\alpha _{k}\left( \left( \eta
\circ \widetilde{\gamma }^{-1}\right) \left( X^{\left( n\right) }\right)
,\sigma _{k}\cdot \nabla X^{\left( n\right) }\right) _{2}d\beta _{k}\left(
s\right) .
\end{eqnarray*}
Keeping in mind that ${div}\ \sigma _{k}=0$ (assertion already used in the previous equality!), we can see that 
\begin{eqnarray*}
\left( \left( \eta \circ \widetilde{\gamma }^{-1}\right) \left( X^{\left(
n\right) }\right) ,\sigma _{k}\cdot \nabla X^{\left( n\right) }\right) _{2}
&=&\left( \nabla \left( \Upsilon \left( X^{\left( n\right) }\right) \right)
\cdot \sigma _{k}\right) _{2}\medskip \medskip \\
&=&-\left( \Upsilon \left( X^{\left( n\right) }\right) ,{div}\ \sigma
_{k}\right) _{2}=0
\end{eqnarray*}%
where $\Upsilon $ is a primitive of $\eta \circ \widetilde{\gamma }^{-1}$. Indeed, by the choice of our basis $e_k$ as having Dirichlet homogeneous conditions, we also inherit these conditions on $\mu_ke_k$, and, as a consequence, we have $N_\xi\cdot(\mu_ke_k)=0$ on the boundary, which comes in handy to simplify Green's formula.

Keeping in mind the definition of $g$ we have that%
\begin{eqnarray*}
&&\int_{0}^{t}\int_{\mathcal{O}}\left( \nabla \left( \left( g\circ 
\widetilde{\gamma }^{-1}\right) \left( X^{\left( n\right) }\right) \right)
\right)^t Q \nabla X^{\left( n\right) }d\xi ds \\
&=&\frac{1}{2}\int_{0}^{t}\int_{\mathcal{O}}\left( \eta ^{\prime }\left( 
\widetilde{\gamma }^{-1}\left( X^{\left( n\right) }\right) \right) \right)
^{2}\left( \widetilde{\gamma }^{-1}\right) ^{\prime }\left( X^{\left(
n\right) }\right) \left(\nabla X^{\left( n\right) }\right)^tQ\ \nabla X^{\left( n\right) }d\xi ds \\
&=&\frac{1}{2}\int_{0}^{t}\int_{\mathcal{O}}\left(\left( \eta \circ \widetilde{%
\gamma }^{-1}\right) ^{\prime }\left( X^{\left( n\right) }\right) 
\nabla X^{\left( n\right) }\right)^t Q \eta ^{\prime
}\left( \widetilde{\gamma }^{-1}\left( X^{\left( n\right) }\right) \right) 
\nabla X^{\left( n\right) }d\xi ds
\end{eqnarray*}

Since we have assumed that $0\leq \left( \widetilde{\gamma }^{-1}\right) ^{\prime
}\leq 1$, we get that 
\begin{eqnarray*}
&&\sum\limits_{k=1}^{n}\int_{0}^{t}\int_{\mathcal{O}}\left\vert
P_{n}\left( \alpha _{k}\sigma _{k}\cdot \nabla \eta \left( \widetilde{\gamma 
}^{-1}\left( X^{\left( n\right) }\right) \right) \right) \right\vert
^{2}d\xi ds \\
&\leq &\sum\limits_{k=1}^{n}\int_{0}^{t}\int_{\mathcal{O}}\left\vert
\left( \alpha _{k}\sigma _{k}\cdot \nabla \eta \left( \widetilde{\gamma }%
^{-1}\left( X^{\left( n\right) }\right) \right) \right) \right\vert ^{2}d\xi
ds \\
&\leq &\int_{0}^{t}\int_{\mathcal{O}}\left( \left( \eta \circ \widetilde{%
\gamma }^{-1}\right) ^{\prime }\left( X^{\left( n\right) }\right) \right)
^{2}\left( \nabla X^{\left( n\right) }\right) ^{t}Q \nabla
X^{\left( n\right) }d\xi ds \\
&\leq &\int_{0}^{t}\int_{\mathcal{O}}\left( \eta \circ \widetilde{\gamma }^{-1}\right) ^{\prime }\left( X^{\left( n\right) }\right) \left( \nabla X^{\left( n\right) }\right) ^{t}Q\left( \eta \circ  \widetilde{\gamma }^{-1}\right) ^{\prime }\left( X^{\left( n\right) }\right) \nabla X^{\left( n\right) }d\xi ds \\
&\leq &2\int_{0}^{t}\int_{\mathcal{O}}\left( \nabla \left( g\circ \widetilde{%
\gamma }^{-1}\right) \left( X^{\left( n\right) }\right) \right) ^{t}Q\left(
\xi \right) \nabla X^{\left( n\right) }d\xi ds,
\end{eqnarray*}%
which leads to%
\begin{equation}
\underset{t\in \left[ 0,T\right] }{\sup }\left\Vert X^{\left( n\right)
}\left( t\right) \right\Vert _{2}^{2}+2\int_{0}^{t}\int_{\mathcal{O}}\Psi
^{\prime }\left( X^{\left( n\right) }\right) \left\vert \nabla X^{\left(
n\right) }\right\vert ^{2}d\xi ds\leq \left\Vert X^{\left( n\right) }\left(
0\right) \right\Vert _{2}^{2},  \label{Ito1}
\end{equation}%
uniformly for all $\omega \in \Omega $. The reader is invited to note that, since $X^{(n)}(0)$ is a mere projection of the initial datum, the right-hand member in \eqref{Ito1} can be bounded uniformly w.r.t. $n\geq 1$.

We deduce that, for every $p>1$, and along some subsequence, 
\begin{eqnarray}
X^{\left( n\right) } &\rightharpoonup &X\text{ weakly in }L^p\left(
0,T;L^{2}\left( \mathcal{O}\right) \right) ,\quad \mathbb{P}\text{-a.s.}
\label{starweak} \\
X^{\left( n\right) } &\rightharpoonup &X\text{ weakly in }L^{2}\left( \left(
0,T\right) ;H_{0}^{1}\left( \mathcal{O}\right) \right) ,\quad \mathbb{P}%
\text{-a.s.}.  \notag
\end{eqnarray}%
Our keen readers might ask why the convergence is not merely weak in spaces like $L^2(\Omega;L^{2}\left( \left(0,T\right) ;H_{0}^{1}\left( \mathcal{O}\right) \right))$ with the probability space still present, instead of $\mathbb{P}-a.s.$. The passage from one to another can be found, for instance, in \cite[Theorem 3.1]{Yannelis}. Furthermore, the fist convergence is actually in $L^2(\Omega;L^\infty\left(
0,T;L^{2}\left( \mathcal{O}\right) \right))$, but the $\mathbb{P}$-a.s. version requires a reflexive space, hence the presence of $L^p$, with an arbitrarily large $p>1$. The $L^\infty$ bounds on $X^{(n)}$ will also be inherited on the limit.\\
By recalling that the functions $\Psi ,~g,~\eta $ and $\widetilde{\gamma }^{-1}$ are Lipschitz-continuous (hence present at most linear growth), we also
have that%
\begin{eqnarray*}
\Psi \left( X^{\left( n\right) }\right) &\rightharpoonup &\varkappa \text{
weakly in }L^{2}\left( \left( 0,T\right) ;L^{2}\left( \mathcal{O}\right)
\right) ,\quad \mathbb{P}\text{-a.s.} \\
g\left( \widetilde{\gamma }^{-1}\left( X^{\left( n\right) }\right) \right)
&\rightharpoonup &\rho \text{ weakly in }L^{2}\left( \left( 0,T\right)
;L^{2}\left( \mathcal{O}\right) \right) ,\quad \mathbb{P}\text{-a.s.} \\
\eta \left( \widetilde{\gamma }^{-1}\left( X^{\left( n\right) }\right)
\right) &\rightharpoonup &\zeta \text{ weakly in }L^{2}\left( \left(
0,T\right) ;L^{2}\left( \mathcal{O}\right) \right) ,\quad \mathbb{P}\text{%
-a.s.}
\end{eqnarray*}

In order to identify the limits of the nonlinear terms, we need to show the
strong convergence of the approximating solutions in $H^{-1}\left( \mathcal{O%
}\right) .$\ To this purpose we introduce the inclusions%
\begin{equation*}
L^{2}\left( \mathcal{O}\right) \subset H^{-1}\left( \mathcal{O}\right)
\subset H^{-\beta }\left( \mathcal{O}\right) ,
\end{equation*}%
where $\beta $ is assumed to be large enough. Note that $L^{2}\left( 
\mathcal{O}\right) \hookrightarrow H^{-1}\left( \mathcal{O}\right) $ is a
compact embedding. We intend to use the compactness result in Corollary 5
from \cite{simon}. According to this result, one has 
\begin{equation*}
L^{\infty }\left( 0,T;L^{2}\left( \mathcal{O}\right) \right) \cap W^{\alpha
,r}\left( 0,T;H^{-\beta }\left( \mathcal{O}\right) \right) \subset C\left( %
\left[ 0,T\right] ;H^{-1}\left( \mathcal{O}\right) \right),
\end{equation*}%
with compact inclusion, provided that $\alpha <\frac{1}{2},~\beta >4$, $r\geq 4$  and $%
\alpha r>1$.

In order to get the strong convergence of the sequence of approximating
solutions in $H^{-1}\left( \mathcal{O}\right) $, we consider the sequence of
probability measures $\left\{ \nu _{n}\right\} _{n}$, where $\nu _{n}$ is
the law of $X^{\left( n\right) }$, and we prove that $\left\{ \nu
_{n}\right\} _{n}$ is tight on the space $C\left( \left[ 0,T\right]
;H^{-1}\left( \mathcal{O}\right) \right) $. This means that for each $%
\varepsilon >0$, we should exhibit a compact subset $K_{R}$ of $C\left( [0,T];H^{-1}\left( \mathcal{O}\right) \right) $ such that $\nu _{n}\left(
K_{R}^{c}\right) \leq \varepsilon $ for all $n\in \mathbb{N}$, and we are searching for these compacts under a particular form suggested by the aforementioned Corollary.

We define for each $R$ the set%
\begin{equation*}
K_{R}=\left\{ f\in C\left( 0,T;H^{-1}\left( \mathcal{O}\right) \right)
;~\left\Vert f\right\Vert _{L^{\infty }\left( 0,T;L^{2}\left( \mathcal{O}%
\right) \right) }+\left\Vert f\right\Vert _{W^{\alpha ,r}\left(
0,T;H^{-\beta }\left( \mathcal{O}\right) \right) }\leq R\right\} .
\end{equation*}

For our readers' convenience, we ecall that the fractional Sobolev space $%
W^{\alpha ,r}\left( 0,T;H^{-\beta }\left( \mathcal{O}\right) \right) $ is
defined as a set of functions $f\in L^{r}\left( 0,T;H^{-\beta }\left( 
\mathcal{O}\right) \right) $ such that 
\begin{equation*}
\int_{0}^{T}\int_{0}^{T}\frac{\left\Vert f\left( t\right) -f\left( s\right)
\right\Vert _{H^{-\beta }\left( \mathcal{O}\right) }^{r}}{\left\vert
t-s\right\vert ^{1+\alpha r}}dtds<\infty .
\end{equation*}

In order to show that the sequence of laws $\left\{ \nu _{n}\right\} _{n}$ is tight w.r.t. the Polish space $%
C\left( \left[ 0,T\right] ;H^{-1}\left( \mathcal{O}\right) \right) $, we will
first provide the following upper bound for the expected value
\begin{equation}
\mathbb{E}\underset{t\in \left[ 0,T\right] }{\sup }\left\vert X^{\left(
n\right) }\left( t\right) \right\vert _{2}^{2}ds+\mathbb{E}%
\int_{0}^{T}\int_{0}^{T}\frac{\left\Vert X^{\left( n\right) }\left( t\right)
-X^{\left( n\right) }\left( s\right) \right\Vert _{H^{-\beta }}^{r}}{%
\left\vert t-s\right\vert ^{1+\alpha r}}dtds\leq C.  \label{estim-tight}
\end{equation}

Since the first term from (\ref{estim-tight}) is bounded from the estimates (%
\ref{Ito1}) above (and this without requiring expectation), it is sufficient to provide and estimate for the second term in the left-hand member.

To this purpose, we rely on the following lemmas.
\end{proof}
\begin{lemma}\label{Lemma1}
There exists a constant $C$ independent of $n$ such that for $r\geq 4$ and
all $0\leq s\leq t\leq T$ it holds%
\begin{equation*}
\mathbb{E}\left[ \left\vert \left( X_{t}^{\left( n\right) }-X_{s}^{\left(
n\right) },e_{j}\right) _{2}\right\vert ^{r}\right] \leq C\left\vert t-s\right\vert ^{\frac{r}{2}}\left( \lambda
_{j}^{r}+\left\Vert {div}\left[ Q\left( \xi \right) \nabla e_{j}%
\right] \right\Vert _{2}^{r}+\lambda _{j}^{\frac{r}{2}}\right) .
\end{equation*}
\end{lemma}
The result is proven by estimating each of the terms appearing in the definition of the solution. For our readers' comfort, we relegate the explicit proof to a specially dedicated Subsection \ref{SubsLem} and continue the proof of our main result.

\begin{lemma}\label{Lemma2}
For $\beta >4$ and  $r\geq 4$ , there is a
constant $C$ independent of $n$ such that%
\begin{equation*}
\mathbb{E}\left[ \left\Vert X_{t}^{\left( n\right) }-X_{s}^{\left( n\right)
}\right\Vert _{H^{-\beta }}^{r}\right] \leq C\left\vert t-s\right\vert ^{%
\frac{r}{2}}.
\end{equation*}
\end{lemma}
The results employs Hölder's inequality and the initial assumptions on $Q$ and, as it has already been the case for the previous result, its complete proof is relegated to Subsection \ref{SubsLem}.

\bigskip

\begin{proof}[Proof of Theorem 3 (continued)]
We can now use the previous Lemmas in order to get 
\begin{equation*}
\int_{0}^{T}\int_{0}^{T}\frac{\mathbb{E}\left\Vert X^{\left( n\right)
}\left( t\right) -X^{\left( n\right) }\left( s\right) \right\Vert
_{H^{-\beta }}^{r}}{\left\vert t-s\right\vert ^{1+\alpha r}}dtds\leq
\int_{0}^{T}\int_{0}^{T}\frac{C\left\vert t-s\right\vert ^{\frac{r}{2}}}{%
\left\vert t-s\right\vert ^{1+\alpha r}}dtds\leq C,
\end{equation*}%
since $1+\alpha r-\frac{r}{2}<1$ for $\alpha \in \left( 0,\frac{1}{2}\right) 
$ and $r\geq 4$.

As we have already explained, at this point we use Corollary 5 from \cite{simon}, to first get that the set%
\begin{equation*}
K_{R}=\left\{ f\in C\left( 0,T;H^{-1}\left( \mathcal{O}\right) \right)
;~\left\Vert f\right\Vert _{L^{\infty }\left( 0,T;L^{2}\left( \mathcal{O}%
\right) \right) }+\left\Vert f\right\Vert _{W^{\alpha ,r}\left(
0,T;H^{-\beta }\left( \mathcal{O}\right) \right) }\leq R\right\}
\end{equation*}%
is a compact set in the space $C\left( \left[ 0,T\right] ;H^{-1}\left( \mathcal{O}%
\right) \right) $.

By using Markov's inequality, we have%
\begin{eqnarray*}
\nu _{n}\left( K_{R}^{c}\right) &=&\mathbb{P}\left( X^{\left( n\right) }\in
K_{R}^{c}\right) \\
&=&\mathbb{P}\left( \left\Vert X^{\left( n\right) }\right\Vert _{L^{\infty
}\left( 0,T;L^{2}\left( \mathcal{O}\right) \right) }+\left\Vert X^{\left(
n\right) }\right\Vert _{W^{\alpha ,r}\left( 0,T;H^{-\beta }\left( \mathcal{O}%
\right) \right) }>R\right) \\
&\leq &\frac{1}{R}\mathbb{E}\left[\left\Vert X^{\left( n\right) }\right\Vert _{L^{\infty
}\left( 0,T;L^{2}\left( \mathcal{O}\right) \right) }+\left\Vert X^{\left(
n\right) }\right\Vert _{W^{\alpha ,r}\left( 0,T;H^{-\beta }\left( \mathcal{O}%
\right) \right) }\right]\leq \frac{C}{R}\leq\varepsilon,
\end{eqnarray*}%
for $R$ sufficiently large and we obtain that the family of laws $\left\{
\nu _{n}\right\} _{n}$ is tight in the space $C\left( \left[ 0,T\right]
;H^{-1}\left( \mathcal{O}\right) \right) $.

As a consequence, by Skorohod's theorem, there exists a probability space $\left( 
\widetilde{\Omega },\widetilde{\mathcal{F}},\widetilde{\mathbb{P}}\right) $
endowed with a filtration $\left( \widetilde{\mathcal{F}}_{t}\right) _{t\in \left[
0,T\right] }$, a sequence of filtrations $\left( \widetilde{\mathcal{F}}%
_{t}^{\left( n\right) }\right) _{t\in \left[ 0,T\right] }$, and the
stochastic processes $\widetilde{X}^{\left( n\right) }$ with the $\left( 
\widetilde{\mathcal{F}}_{t}^{\left( n\right) }\right) _{t\in \left[ 0,T%
\right] }$ cylindrical Wiener process $\widetilde{W}^{\left( n\right)
}=\sum\limits_{k=1}^{\infty }e_{k}\widetilde{\beta }_{t}^{k,n}$ and also $%
\widetilde{X}$ with the $\left( \widetilde{\mathcal{F}}_{t}\right) _{t\in %
\left[ 0,T\right] }$ cylindrical Wiener process $\widetilde{W}%
=\sum\limits_{k=1}^{\infty }e_{k}\widetilde{\beta }_{t}^{k}$ on $\left( 
\widetilde{\Omega },\widetilde{\mathcal{F}},\widetilde{\mathbb{P}}\right) $.
Furthermore, the law of $\widetilde{X}^{\left( n\right) }$ is the same as
the law of $X^{\left( n\right) },$ the law of $\widetilde{W}^{\left(
n\right) }$is the same as the law of $W^{\left( n\right) }$and 
\begin{eqnarray}
\widetilde{X}^{\left( n\right) } &\longrightarrow &\widetilde{X}\text{
strongly in }C\left( \left[ 0,T\right] ;H^{-1}\left( \mathcal{O}\right)
\right) ,~\widetilde{\mathbb{P}}- a.s.  \label{star} \\
\widetilde{\beta }_{t}^{k,n} &\longrightarrow &\widetilde{\beta }_{t}^{k}%
\text{ strongly in }C\left( \left[ 0,T\right] ;\mathbb{R}\right) ,~\widetilde{\mathbb{P}}-%
a.s.,~\forall k\geq 0,  \notag
\end{eqnarray}%
as $n\rightarrow \infty $. 

In order to conclude that $\left( \widetilde{X},\widetilde{W}\right) $ is a
weak solution to our equation of interest, we still need to check an adequate
integrability property as follows.
\end{proof}

\begin{lemma}\label{Lemma3}
The process $\widetilde{X}$ has $\widetilde{\mathbb{P}}$-a.s. weakly
continuous trajectories w.r.t. $L^{2}\left( \mathcal{O}\right) $, and it further satisfies%
\begin{equation*}
\underset{t\in \left[ 0,T\right] }{\sup }\left\Vert \widetilde{X}\left(
t\right) \right\Vert _{2}^{2}\leq \left\Vert X_{0}\right\Vert _{2}^{2},\quad 
\widetilde{\mathbb{P}}-a.s.
\end{equation*}
\end{lemma}
The proof of this Lemma follows a similar argument with Lemma 3.5 from \cite%
{Flandoli2} and for this reason we shall only hint to the main ideas in Subsection \ref{SubsLem}. The interested reader is invited to consult the extensive initial proof.

\medskip

\begin{proof}[Proof of Theorem 3 (continued)]
We consider the limit in the (PDE) weak formulation of approximating solutions, i.e.,%
\begin{eqnarray*}
\left( \widetilde{X}^{\left( n\right) }\left( t\right) ,e_{j}\right) _{2}
&=&\left( x^{\left( n\right) },e_{j}\right) _{2}+\int_{0}^{t}\int_{\mathcal{O%
}}\Psi \left( \widetilde{X}^{\left( n\right) }\left( s\right) \right) \Delta
e_{j}d\xi ds \\
&&+\int_{0}^{t}\int_{\mathcal{O}}g\left( \widetilde{\gamma }%
^{-1}\left( \widetilde{X}^{\left( n\right) }\right) \right) {div}\left[
Q \nabla e_{j}\right] d\xi ds \\
&&+\sum_{k=1}^{n }\int_{0}^{t}\alpha _{k}\left( \eta \left( \widetilde{%
\gamma }^{-1}\left( \widetilde{X}^{\left( n\right) }\right) \right) ,\sigma
_{k}\cdot \nabla e_{j}\right) _{2}d\widetilde{\beta }_{s}^{k,n},\text{ }%
\widetilde{\mathbb{P}}-a.s.
\end{eqnarray*}

Since $e_{j}\in H_{0}^{1}\left( \mathcal{O}\right) \cap H^{2}\left( \mathcal{%
O}\right) $, we can write 
\begin{equation*}
\left( \widetilde{X}^{\left( n\right) }\left( t\right) ,e_{j}\right)
_{2}=~_{H^{-1}\left( \mathcal{O}\right) }\left( \widetilde{X}^{\left(
n\right) }\left( t\right) ,e_{j}\right) _{H_{0}^{1}\left( \mathcal{O}\right)
}
\end{equation*}%
and the strong convergence in $H^{-1}\left( \mathcal{O}\right) $ implies that%
\begin{equation*}
_{H^{-1}\left( \mathcal{O}\right) }\left( \widetilde{X}^{\left( n\right)
}\left( \cdot\right) ,e_{j}\right) _{H_{0}^{1}\left( \mathcal{O}\right)
}\rightarrow _{H^{-1}\left( \mathcal{O}\right) }\left( \widetilde{X}\left(
\cdot\right) ,e_{j}\right) _{H_{0}^{1}\left( \mathcal{O}\right) }\text{ in }%
C\left( \left[ 0,T\right] ;\mathbb{R}\right) .
\end{equation*}

The reader is invited to keep in mind that 
\begin{equation*}
\left\vert \left( \widetilde{X}^{\left( n\right) }\left( t\right)
,e_{j}\right) _{2}\right\vert \leq \left\Vert \widetilde{X}^{\left( n\right)
}\left( t\right) \right\Vert _{2}\left\Vert e_{j}\right\Vert _{2}\leq
C\left\Vert x^{\left( n\right) }\right\Vert _{2},\textnormal{ for all }t\in \left[ 0,T\right], \ \widetilde{\mathbb{P}}%
\text{-a.s.}  \text{,}
\end{equation*}%
(as usual, with a generic constant $C>0$ that depends neither on $j$, nor on the time parameter $t\in[0,T]$, nor on the approximating parameter $n\geq 1$), such that
\begin{equation*}
\underset{n\rightarrow \infty }{\lim }\left( \widetilde{X}^{\left( n\right)
}\left( t\right) ,e_{j}\right) _{2}=\left( \widetilde{X}\left( t\right)
,e_{j}\right) _{2}\text{,}\textnormal{ for all }t\in \left[ 0,T\right], \ \widetilde{\mathbb{P}}%
\text{-a.s.}  \text{,}
\end{equation*}
and, by a dominated convergence argument, we actually have the aforementioned convergence in $L^p([0,T];\mathbb{R}),\widetilde{\mathbb{P}}-a.s.$ (for every $p\geq 2$).
By the convergences (\ref{star}) we can pass to the limit in the second and
the third terms from the right hand side, without identifying the limit for
the moment%
\begin{eqnarray*}
\int_{0}^{t}\int_{\mathcal{O}}\Psi \left( \widetilde{X}^{\left( n\right)
}\left( s\right) \right) \Delta e_{j}d\xi ds &\longrightarrow
&\int_{0}^{t}\int_{\mathcal{O}}\varkappa \Delta e_{j}d\xi ds\text{,} \\
\int_{0}^{t}\int_{\mathcal{O}}g\left( \widetilde{\gamma }^{-1}\left( 
\widetilde{X}^{\left( n\right) }\right) \right) {div}\left[ Q\left( \xi
\right) \nabla e_{j}\right] d\xi ds &\longrightarrow &\int_{0}^{t}\int_{%
\mathcal{O}}\rho {div}\left[ Q\left( \xi \right) \nabla e_{j}\right]
d\xi ds\text{,}
\end{eqnarray*}%
$\widetilde{\mathbb{P}}$-a.s. for all $t\in \left[ 0,T\right] .$

In order to identify the limit, we can use the fact that $\Psi $ and $g\circ 
\widetilde{\gamma }^{-1}$ are maximal monotone functions on $\mathbb{R}$.
More precisely, this implies that $\Delta \Psi $ and $\Delta \left( g\circ 
\widetilde{\gamma }^{-1}\right) $ are maximal monotones in $H^{-1}\left( 
\mathcal{O}\right) $ and we can use the weak-strong convergence property in
order to identify their limits in $H^{-1}\left( \mathcal{O}\right) $.

From the Lipschitz property of $\Psi $ and $g\circ \widetilde{\gamma }^{-1}$
and using (\ref{Ito1}) we have that 
\begin{equation*}
\left( \Delta \Psi \left( \widetilde{X}^{\left( n\right) }\left( s\right)
\right) \right) _{n\in \mathbb{N}}\text{ and }\left( \Delta \left( g\circ 
\widetilde{\gamma }^{-1}\right) \left( \widetilde{X}^{\left( n\right)
}\left( s\right) \right) \right) _{n\in \mathbb{N}}
\end{equation*}%
are bounded in $L^{2}\left( \left( 0,T\right) ;H^{-1}\left( \mathcal{O}%
\right) \right) ,\quad \mathbb{P}-a.s.$. Therefore, we are able to exhibit some subsequences (still denoted by $n$ for simplicity), that are weakly convergent in $L^{2}\left( \left( 0,T\right) ;H^{-1}\left( 
\mathcal{O}\right) \right) ,~\mathbb{P}$-a.s. Combining this with the strong
convergence in $L^{2}\left( \left( 0,T\right) ;H^{-1}\left( \mathcal{O}%
\right) \right) ,~\mathbb{P}$-a.s from (\ref{star}) we can identify the
limits.

We can show now the convergence of the stochastic integrals. In order to
apply Lemma 8 from \cite{Flandoli1}, let us denote by
\begin{equation*}
G^{n}=\sum_{k=1}^{n}\alpha _{k}\left( \eta \left( \widetilde{\gamma }%
^{-1}\left( \widetilde{X}^{\left( n\right) }\right) \right) ,\sigma
_{k}\cdot \nabla e_{j}\right) _{2},
\end{equation*}%
and 
\begin{equation*}
G=\sum_{k=1}^{\infty }\alpha _{k}\left( \eta \left( \widetilde{\gamma }%
^{-1}\left( \widetilde{X}\right) \right) ,\sigma _{k}\cdot \nabla
e_{j}\right) _{2},
\end{equation*}%
and it's sufficient to prove that%
\begin{equation*}
G^{n}\longrightarrow G\text{ in }L^{2}\left( 0,T;\mathbb{R}\right) ,\text{ }%
\mathbb{P}\text{-a.s.}
\end{equation*}

We have%
\begin{eqnarray*}
&&\int_{0}^{t}\sum_{k=1}^{n}\alpha _{k}\left( \eta \left( \widetilde{\gamma 
}^{-1}\left( \widetilde{X}^{\left( n\right) }\right) \right) -\eta \left( 
\widetilde{\gamma }^{-1}\left( \widetilde{X}\right) \right) ,\sigma
_{k}\cdot \nabla e_{j}\right) _{2}^{2}ds \\
&=&\int_{0}^{t}\left( \left( \eta \left( \widetilde{\gamma }^{-1}\left( 
\widetilde{X}^{\left( n\right) }\right) \right) -\eta \left( \widetilde{%
\gamma }^{-1}\left( \widetilde{X}\right) \right) \right) ^{2}\left( \nabla
e_{j}\right) ^{t},Q \nabla e_{j}\right) _{2}ds \\
&=&\int_{0}^{t}~_{H_{0}^{1}\left( \mathcal{O}\right) }\left( \left( \eta
\left( \widetilde{\gamma }^{-1}\left( \widetilde{X}^{\left( n\right)
}\right) \right) -\eta \left( \widetilde{\gamma }^{-1}\left( \widetilde{X}%
\right) \right) \right) ^{2},\left( \nabla e_{j}\right) ^{t}Q
\nabla e_{j}\right) _{H^{-1}\left( \mathcal{O}\right) }.
\end{eqnarray*}

By the same argument as in the identification of the limits from above, one easily shows that
\begin{equation*}
\eta \left( \widetilde{\gamma }^{-1}\left( \widetilde{X}^{\left( n\right)
}\right) \right) \longrightarrow \eta \left( \widetilde{\gamma }^{-1}\left( 
\widetilde{X}\right) \right) \text{, in }L^{2}\left( 0,T;H_{0}^{1}\left( 
\mathcal{O}\right) \right) ,\text{ }\mathbb{P}\text{-a.s.}
\end{equation*}

From the Sobolev embedding theorem we have that $H_{0}^{1}\left( \mathcal{O}%
\right) \subset L^{4}\left( \mathcal{O}\right) $, and, therefore%
\begin{equation*}
\left( \eta \left( \widetilde{\gamma }^{-1}\left( \widetilde{X}^{\left(
n\right) }\right) \right) -\eta \left( \widetilde{\gamma }^{-1}\left( 
\widetilde{X}\right) \right) \right) ^{2}\longrightarrow 0\text{, weakly in }%
L^{2}\left( 0,T;L^{2}\left( \mathcal{O}\right) \right) ,\text{ }\mathbb{P}%
\text{-a.s.}
\end{equation*}%
Keeping in mind that $\left( \nabla e_{j}\right) ^{t}Q\nabla
e_{j}\in L^{2}\left( \mathcal{O}\right) $ we can pass to the limit in
probability in the stochastic integral and apply Lemma 8 from \cite%
{Flandoli1}. Since the convergence in probability implies convergence almost
sure on a subsequence, by reasoning on this subsequence (still indexed by $n$ to make the arguments easier to read), we obtain that 
\begin{eqnarray*}
\left( \widetilde{X}\left( t\right) ,e_{j}\right) _{2} &=&\left(
x,e_{j}\right) _{2}+\int_{0}^{t}\int_{\mathcal{O}}\Psi \left( \widetilde{X}%
\left( s\right) \right) \Delta e_{j}d\xi ds \\
&&+\int_{0}^{t}\int_{\mathcal{O}}g\left( \widetilde{\gamma }%
^{-1}\left( \widetilde{X}\right) \right) {div}\left[ Q\left( \xi
\right) \nabla e_{j}\right] d\xi ds \\
&&+\sum_{k=1}^{\infty }\int_{0}^{t}\alpha _{k}\left( \eta \left( \widetilde{%
\gamma }^{-1}\left( \widetilde{X}\right) \right) ,\sigma _{k}\cdot \nabla
e_{j}\right) _{2}d\widetilde{\beta }_{s}^{k},\text{ }\mathbb{P}-a.s.
\end{eqnarray*}%
and the proof is complete.
\end{proof}
\subsection{Proofs of the Intermediate Lemmas \ref{Lemma1}, \ref{Lemma2}}\label{SubsLem}
\begin{proof}[Proof of Lemma \ref{Lemma1}]
The definition of the solution yields
\begin{equation}\label{est1}
\begin{split}
\left( X_{t}^{\left( n\right) }-X_{s}^{\left( n\right) },e_{j}\right) _{2}
=&\int_{s}^{t}\int_{\mathcal{O}}\left( \Psi \left( X_{l}^{\left( n\right)
}\right) \right) \Delta e_{j}d\xi dl  \\
&+\int_{s}^{t}\int_{\mathcal{O}}g\left( \widetilde{\gamma }%
^{-1}\left( X_{l}^{\left( n\right) }\right) \right) div\left[ Q\nabla e_{j}\right] d\xi dl \\
&+\sum_{k=1}^{\infty }\int_{s}^{t}\alpha _{k}\left( \eta \left( \widetilde{%
\gamma }^{-1}\left( X_{l}^{\left( n\right) }\right) \right) ,\sigma
_{k}\cdot \nabla e_{j}\right) _{2}d\beta _{k}\left( l\right)\\
=&I_{s,t}^{1}+I_{s,t}^{2}+I_{s,t}^{3}. 
\end{split}
\end{equation}
We now turn to the first term and write
\begin{eqnarray*}
\mathbb{E}\left[ \left\vert I_{s,t}^{1}\right\vert ^{r}\right] &= &%
\mathbb{E}\left[ \left\vert \int_{s}^{t}\int_{\mathcal{O}}\left( \Psi \left(
X_{l}^{\left( n\right) }\right) \right) \Delta e_{j}d\xi dl\right\vert ^{r}%
\right] \\
&\leq &\mathbb{E}\left[ \left\vert \int_{s}^{t}\left\Vert \Psi \left(
X_{l}^{\left( n\right) }\right) \right\Vert _{2}\left\Vert \Delta
e_{j}\right\Vert _{\infty }dl\right\vert ^{r}\right] \\
&\leq &C\lambda _{j}^{r}\mathbb{E}\left[ \left\vert \int_{s}^{t}\left\Vert
\Psi \left( X_{l}^{\left( n\right) }\right) \right\Vert _{2}dl\right\vert
^{r}\right] \\
&\leq &C\lambda _{j}^{r}\mathbb{E}\left[ \left\vert \int_{s}^{t}\left\Vert
X_{l}^{\left( n\right) }\right\Vert _{2}dl\right\vert ^{r}\right] \\
&\leq &C\lambda _{j}^{r}\left\vert t-s\right\vert ^{r}.
\end{eqnarray*}
As we have aleready mentioned after \eqref{Ito1}, the $C$ constant can be chosen independently of $n\geq 1$ (and, of course, of the $j$ index, as well as the time parameters $t,s$).
For the second term, we have, owing to classical Cauchy-Schwarz inequalities,%
\begin{eqnarray*}
\mathbb{E}\left[ \left\vert I_{s,t}^{2}\right\vert ^{r}\right] &= &%
\mathbb{E}\left[ \left\vert \int_{s}^{t}\int_{\mathcal{O}}g\left( 
\widetilde{\gamma }^{-1}\left( X_{l}^{\left( n\right) }\right) \right) {%
div}\left[ Q \nabla e_{j}\right] d\xi dl\right\vert ^{r}%
\right] \\
&\leq &\mathbb{E}\left[ \left\vert \int_{s}^{t}\left\Vert
g\left( \widetilde{\gamma }^{-1}\left( X_{l}^{\left( n\right) }\right)
\right) \right\Vert _{2}\left\Vert {div}\left[ Q
\nabla e_{j}\right] \right\Vert _{2}dl\right\vert ^{r}\right] ,
\end{eqnarray*}%
and by using similar arguments as for the first term (i.e., invoking the Lipschitz continuity of $g\circ\tilde{\gamma}^{-1}$ yielding $\mid g\circ\tilde{\gamma}^{-1}(x)\mid\leq \mid g\circ\tilde{\gamma}^{-1}(0)\mid+C\mid x\mid$), we obtain that%
\begin{equation*}
\mathbb{E}\left[ \left\vert I_{s,t}^{2}\right\vert ^{r}\right] \leq
C\left\vert t-s\right\vert ^{r}\left\Vert {div}\left[ Q\nabla e_{j}\right] \right\Vert _{2}^{r}.
\end{equation*}

Finally, for the last term, by employing Burkholder-Davis-Gundy inequality and further arguing similarly, we have%
\begin{eqnarray*}
\mathbb{E}\left[ \left\vert I_{s,t}^{3}\right\vert ^{r}\right] &= &%
\mathbb{E}\left[ \left\vert \sum_{k=1}^{n }\alpha
_{k}\int_{s}^{t}\left( \eta \left( \widetilde{\gamma }^{-1}\left(
X_{l}^{\left( n\right) }\right) \right) ,\sigma _{k}\cdot \nabla
e_{j}\right) _{2}d\beta _{k}\left( l\right) \right\vert ^{r}\right] \\
&\leq &C\mathbb{E}\left[ \left\{\int_{s}^{t}\left( \eta \left( \widetilde{\gamma }%
^{-1}\left( X_{l}^{\left( n\right) }\right) \right) ,\sum_{k=1}^{n
}\alpha _{k}\sigma _{k}\cdot \nabla e_{j}\right) _{2}^{2}dl\right\} ^{\frac{r%
}{2}}\right] \\
&\leq &C\mathbb{E}\left[\left\{ \int_{s}^{t}\left\Vert \eta \left( \widetilde{\gamma 
}^{-1}\left( X_{l}^{\left( n\right) }\right) \right) \right\Vert
_{2}^{2}\left\Vert \sum_{k=1}^{\infty }\alpha _{k}\sigma _{k}\cdot \nabla
e_{j}\right\Vert _{2}^{2}dl\right\} ^{\frac{r}{2}} \right]\\
&\leq &C\left\vert t-s\right\vert ^{\frac{r}{2}}\left\Vert
\sum_{k=1}^{\infty }\alpha _{k}\sigma _{k}\cdot \nabla e_{j}\right\Vert
_{2}^{r} \\
&\leq &C\left\vert t-s\right\vert ^{\frac{r}{2}}\lambda _{j}^{\frac{r}{2}}.
\end{eqnarray*}%
Recall that the series above is convergent owing to the assumptions we imposed at the beginning.

Going back to estimate (\ref{est1}), we obtain that,%
\begin{eqnarray*}
&&\mathbb{E}\left[ \left\vert \left( X_{t}^{\left( n\right) }-X_{s}^{\left(
n\right) },e_{j}\right) _{2}\right\vert ^{r}\right] \medskip \\
&\leq &C\left\vert t-s\right\vert ^{\frac{r}{2}}\left( \lambda
_{j}^{r}+\left\Vert {div}\left[ Q\left( \xi \right) \nabla e_{j}%
\right] \right\Vert _{2}^{r}+\lambda _{j}^{\frac{r}{2}}\right) .
\end{eqnarray*}
\end{proof}

\medskip

\begin{proof}[Proof of Lemma \ref{Lemma2}]
Let $\varepsilon\in (0,1)$. Using Hölder's inequality with $\frac{r}{2},\ \frac{r}{r-2}$, we get
\begin{eqnarray*}
\mathbb{E}\left[ \left\Vert X_{t}^{\left( n\right) }-X_{s}^{\left(
n\right) }\right\Vert _{H^{-\beta }}^{r}\right] \medskip 
&=&\mathbb{E}\left[ \left(\sum_{j=1}^{\infty }\frac{\left( X_{t}^{\left( n\right)
}-X_{s}^{\left( n\right) },e_{j}\right) _{2}^{2}}{\lambda _{j}^{\beta }}%
\right) ^{\frac{r}{2}}\right]\medskip \\
&= &\mathbb{E}\left[ \left(\sum_{j=1}^{\infty }\frac{1}{\lambda _{j}^{\frac{%
\left( 1+\varepsilon \right) \left( r-2\right) }{r}}}\frac{\left(
X_{t}^{\left( n\right) }-X_{s}^{\left( n\right) },e_{j}\right) _{2}^{2}}{%
\lambda _{j}^{\beta -\frac{\left( 1+\varepsilon \right) \left( r-2\right) }{r%
}}}\right) ^{\frac{r}{2}}\right]\medskip \\
&\leq &\left( \sum_{j=1}^{\infty }\frac{1}{\lambda _{j}^{\left(
1+\varepsilon \right) }}\right) ^{\frac{r-2}{2}}\left( \sum_{j=1}^{\infty }%
\frac{\mathbb{E}\left( X_{t}^{\left( n\right) }-X_{s}^{\left( n\right)
},e_{j}\right) _{2}^{r}}{\lambda _{j}^{\frac{r\beta }{2}-\frac{\left(
1+\varepsilon \right) \left( r-2\right) }{\textcolor{red}{2}}}}\right) \\
&\leq &C\left\vert t-s\right\vert ^{\frac{r}{2}}\left( \sum_{j=1}^{\infty }%
\frac{\left( \lambda _{j}^{r}+\left\Vert div\left[ Q\ \nabla e_{j}\right] \right\Vert _{2}^{r}+\lambda _{j}^{%
\frac{r}{2}}\right) }{\lambda _{j}^{\frac{r\beta }{2}-\frac{\left(
1+\varepsilon \right) \left( r-2\right) }{\textcolor{red}{2}}}}\right) .
\end{eqnarray*}

Keeping in mind the assumptions on the operator $Q$ we have that 
\begin{equation*}
\left\Vert div\left[ Q \nabla e_{j}\right]
\right\Vert _{2}^{r}\leq \gamma ^{\frac{r}{2}}\left( \left\Vert
e_{j}\right\Vert _{H_{0}^{1}\left( \mathcal{O}\right) }^{2}+\left\Vert
e_{j}\right\Vert _{H^{2}\left( \mathcal{O}\right) }^{2}\right) ^{\frac{r}{2}%
}\leq C\gamma ^{\frac{r}{2}}\left(\lambda _{j}^{\frac{r}{2}}+\lambda _{j}^{r}\right),
\end{equation*}%
and therefore%
\begin{eqnarray*}
\mathbb{E}\left[ \left\Vert X_{t}^{\left( n\right) }-X_{s}^{\left( n\right)
}\right\Vert _{H^{-\beta }}^{r}\right] &\leq &C\left\vert t-s\right\vert ^{%
\frac{r}{2}}\left( \sum_{j=1}^{\infty }\frac{\lambda _{j}^{r}+\lambda _{j}^{\frac{r}{2}}}{\lambda _{j}^{%
\frac{r\beta }{2}-\frac{\left( 1+\varepsilon \right) \left( r-2\right) }{\textcolor{red}{2}}}}%
\right) \medskip \\
&\leq &C\left\vert t-s\right\vert ^{\frac{r}{2}},
\end{eqnarray*}%
since $\frac{r\beta }{2}-\frac{\left( 1+\varepsilon \right) \left(
r-2\right) }{2}-r>1+\varepsilon $ for $\beta >4,~\varepsilon \in \left(
0,1\right) $ and $r\geq 4.$
\end{proof}

\medskip

\begin{proof}[Sketch of the Proof of Lemma \ref{Lemma3}]
As we have already mentioned, we borrow the arguments from Lemma 3.5 from \cite%
{Flandoli2} and we only hint to the main ideas.

Since $\widetilde{X}^{\left( n\right) }$ has the same law as $X^{\left(
n\right) }$ we have from the previous estimates that 
\begin{equation*}
\underset{t\in \left[ 0,T\right] }{\sup }\left\Vert \widetilde{X}^{\left(
n\right) }\left( t\right) \right\Vert _{2}^{2}\leq \left\Vert
X_{0}\right\Vert _{2}^{2},\quad \widetilde{\mathbb{P}}-a.s
\end{equation*}%
and therefore there exist a set $\Gamma \subset \widetilde{\Omega }$ of full
measure such that, for every $\omega \in \Gamma $ we have 
\begin{equation*}
\underset{t\in \left[ 0,T\right] }{\sup }\left\Vert \widetilde{X}^{\left(
n\right) }\left( t,\omega \right) \right\Vert _{2}^{2}\leq \left\Vert
X_{0}\right\Vert _{2}^{2},\quad \widetilde{\mathbb{P}}-a.s.
\end{equation*}

Let us fix $\omega \in \Gamma $ and by the relation above we have that the
sequence $\left\{ \widetilde{X}^{\left( n\right) }\left( t,\omega \right)
\right\} _{n}$ in bounded in $L^{\infty }\left( 0,T;L^{2}\left( \mathcal{O}%
\right) \right) $ and so we can extract a subsequence which is weak*
convergent in $L^{\infty }\left( 0,T;L^{2}\left( \mathcal{O}\right) \right) $
and therefore also in $L^{\infty }\left( 0,T;H^{-1}\left( \mathcal{O}\right)
\right) $ which implies by (\ref{star}) that the limit has to be $\widetilde{%
X}$ and 
\begin{equation*}
\left\Vert \widetilde{X}\left( \cdot ,\omega \right) \right\Vert _{L^{\infty
}\left( 0,T;L^{2}\left( \mathcal{O}\right) \right) }\leq \underset{n}{\lim
\inf }\left\Vert \widetilde{X}^{\left( n\right) }\left( \cdot ,\omega
\right) \right\Vert _{L^{\infty }\left( 0,T;L^{2}\left( \mathcal{O}\right)
\right) }\leq \left\Vert X_{0}\right\Vert _{2}^{2},
\end{equation*}

In particular there exist a subset $S_{\omega }\subset \left[ 0,T\right] $
of full Lebesgue measure such that $\left\Vert \widetilde{X}\left( s,\omega
\right) \right\Vert _{2}^{2}\leq \left\Vert X_{0}\right\Vert _{2}^{2}$ for
every $s\in S_{\omega }.$ Finally, for $t\in \left[ 0,T\right] \setminus
S_{\omega }$ and a sequence $t_{n}\rightarrow t$ with $t_{n}\in S_{\omega }$
we get by classical theorey that 
\begin{equation*}
\left\Vert \widetilde{X}\left( t,\omega \right) \right\Vert _{L^{\infty
}\left( 0,T;L^{2}\left( \mathcal{O}\right) \right) }\leq \underset{n}{\lim
\inf }\left\Vert \widetilde{X}\left( t_{n},\omega \right) \right\Vert
_{L^{\infty }\left( 0,T;L^{2}\left( \mathcal{O}\right) \right) }\leq
\left\Vert X_{0}\right\Vert _{2}^{2},
\end{equation*}%
and therefore%
\begin{equation*}
\underset{t\in \left[ 0,T\right] }{\sup }\left\Vert \widetilde{X}\left(
t,\omega \right) \right\Vert _{2}^{2}\leq \left\Vert X_{0}\right\Vert
_{2}^{2},\quad \widetilde{\mathbb{P}}-a.s.
\end{equation*}

By the same argument as in \cite{Flandoli2} we have that $\widetilde{X}$ has 
$\widetilde{\mathbb{P}}$-a.s. weakly continuous trajectories in $L^{2}\left( 
\mathcal{O}\right) $ and the proof of the Lemma is complete.
\end{proof}
\section*{Acknowledgments}
I.C. was supported in parts by the DEFHY3GEO project funded by Région Normandie and European Union with ERDF fund (convention 21E05300).\\
The research of F.F. is funded by the European Union, ERC NoisyFluid, No. 101053472.\\
D.G. acknowledges financial support from National Sciences and Engineering Research Council (NSERC), Canada, 
Grant/Award Number: RGPIN-2025-03963.
\bigskip

\subsection*{Submission Statement}
\noindent The work presented here has not been published previously,  it is not under consideration for publication elsewhere.\\
The publication is approved by all authors and by the responsible authorities where the work was carried out. If accepted, it will not be published elsewhere in the same form, in English or in any other language, including electronically without the written consent of the copyright-holder.
\subsection*{Declaration of Interest}
The authors have no competing interest to declare.
\subsection*{Declaration of Generative AI in Scientific Writing}
The paper makes no use of generative AI.
\subsection*{Author Contributions}
\textbf{Ioana Ciotir}: Formal analysis; Funding acquisition; Investigation; Methodology; Writing-original draft;\\
\textbf{Franco Flandoli}: Formal analysis; Funding acquisition; Investigation; Methodology; Writing-original draft;\\
\textbf{Dan Goreac}: Formal analysis; Funding acquisition; Investigation; Methodology; Writing-original draft.

\subsection*{Data Availability Statement}
\noindent Data sharing is not applicable to this article as no new data were created or analyzed in this study.

\end{document}